\documentclass[12pt]{article}
\usepackage{caption}
\usepackage{url}
\usepackage{color}
\usepackage{amssymb}
\usepackage{amsmath}
\usepackage{amsthm}
\usepackage{graphicx}
\usepackage{mathscinet}

\newtheorem{theorem}{Theorem}[section]
\newtheorem{lemma}[theorem]{Lemma}

\newtheorem{prop}[theorem]{Proposition}
\newtheorem{remark}[theorem]{Remark}
\newtheorem{counterex}[theorem]{Counterexample}

\title{Resolving the Kohayakawa–-Kreuter Conjecture for Families}
\author{Matthew P. Yancey \thanks{Institute for Defense Analyses / Center for Computing Sciences (IDA / CCS), mpyance@super.org}}

\begin{document}
\maketitle

\begin{abstract}
A graph $G$ is $(a,b)$-sparse if every nonempty subgraph $H$ satisfies $e(H) \leq a v(H) + b$.
We are interested in the conditions under which an $(a,b)$-sparse graph can be partitioned $E(G) = E(G_1) \cup E(G_2)$ such that for $i \in \{1,2\}$ we have that $G_i$ is $(a_i, b_i)$-sparse.
Kuperwasser, Samotij, and Wigderson conjectured that a $(m,0)$-sparse graph can be partitioned into a $(1,-1)$-sparse graph and a $(m,1-2m)$-sparse graph.
We prove the conjecture in full.

The Kohayakawa--Kreuter Conjecture for Families claims that $n^{-1/m_2}$ is the threshold function for the random graph being Ramsey a.a.s. for graph families $\mathcal{H}_1, \ldots , \mathcal{H}_r$.
Kuperwasser, Samotij, and Wigderson motivated their conjecture by proving that it is sufficient to establish the Kohayakawa--Kreuter Conjecture for Families.
\end{abstract}

\section{Introduction}

In this paper we deal with simple graphs.
For a graph $G$, let $v(G)$ and $e(G)$ denote the number of vertices and edges it contains, respectively.
For numbers $a \in \mathbb{R}_+, b \in \mathbb{R}$, we say that a graph $G$ is $(a,b)$-sparse if every nonempty subgraph $H$ satisfies $e(H) \leq a v(H) + b$.
As examples, the family of forests is characterized as the $(1,-1)$-sparse graphs, all planar graphs are $(3,-6)$-sparse, while a $K_{3,3}$ with it's edges subdivided an arbitrary number of times is $(1,3)$-sparse.
For background on the applications of sparsity\footnote{They use slightly different notation, where $(a,b)$-sparse means $e(H) \leq a v(H) - b$, while our notation is consistent with \cite{ChoiKostochkaYancey,KostochkaXuZhu}.}, see the introduction by Lee and Streinu \cite{LEE20081425}, with additional references given in the more recent \cite{CY2020} (see Section 2.3).
The case $2a + b < 1$ is considered pathological, as only the empty graph can satisfy the condition.
We are interested in the question on when an $(a,b)$-sparse graph $G$ has a partition $E(G) = E(G_1) \cup E(G_2)$ such that $G_i$ is $(a_i, b_i)$-sparse for $i \in \{0,1\}$.

Clearly, if $G$ partitions $E(G) = E(G_1) \cup E(G_2)$ such that $G_i$ is $(a_i, b_i)$-sparse, then $G$ is $(a_1 + a_2, b_1 + b_2)$-sparse. 
In some cases this necessary condition is sufficient, which can be seen through a connection with matroids.
For a graph $G$ and fixed $a,b$, let $\mathcal{I}_{a,b}$ denote the set of subgraphs of $G$ that are $(a,b)$-sparse. 
Lorea \cite{LOREA1979} proved for integral $a,b$ with $-2a \leq b \leq 0$ that $(E(G), \mathcal{I}_{a,b})$ forms a matroid.
This is generalized to include positive $b$ by White and Whiteley (the statement with proof is Proposition A.1.1 in \cite{whiteley1996some}, who attribute it to \cite{whiteleyWhite}).
Let $\mathcal{A} + \mathcal{B} = \{G_a \cup G_b : G_a \in \mathcal{A}, G_b \in \mathcal{B}\}$.
Pym and Perfect \cite{PYM19701} used submodular functions to compare $\mathcal{I}_{a_1 + a_2,b_1 + b_2}$ against $\mathcal{I}_{a_1, b_1} + \mathcal{I}_{a_2,b_2}$ (observe that they will be equal if and only if the necessary condition is sufficient), but the conclusions are opaque\footnote{To illustrate the inscrutable nature: Example A.2.4 of \cite{whiteley1996some} uses Pym an Perfect's work to conclude that $\mathcal{I}_{2,0} = \mathcal{I}_{1,1} + \mathcal{I}_{1,-1}$, which we prove is not true in Counterexample \ref{disconnected counterexample}.}.
Our first contribution is an explicit class of examples where the necessary condition is sufficient, and in Section \ref{sec: partitioning} we provide a variety of constructions showing that none of the assumptions can be expanded without additional restrictions.

\begin{theorem}\label{partitioning theorem}
Let $a_1,a_2,b_1,b_2$ be integral, with $a_1,a_2$ positive.
Every $(a_1 + a_2, b_1 + b_2)$-sparse graph can be partitioned $E(G) = E(G_1) \cup E(G_2)$ such that $G_1$ is $(a_1, b_1)$-sparse and $G_2$ is $(a_2, b_2)$-sparse if 
\begin{enumerate}
	\item[(A)] $-a_1 \leq b_1 \leq 0$ and $-a_2 \leq b_2 \leq 0$, or
	\item[(B)] $b_1 \geq 0$ and $b_2 \geq 0$.
\end{enumerate}
\end{theorem}

Theorem \ref{partitioning theorem} generalizes results like the Nash-Williams Theorem \cite{NashWilliams}, which says that a $(k,-k)$-sparse graph can be partitioned into $k$ forests; or a special case of a theorem by Hakimi \cite{HAKIMI1965290}, which says that a $(k,0)$-sparse graph can be partitioned into $k$ pseduoforests, where a pseudoforest is a $(1,0)$-sparse graph.

The question gets messier when we consider nonintegral values.
The family of $(\frac{d}{d+1}, 0)$-sparse graphs corresponds to forests where every component contains at most $d$ edges; in particular matchings are $(0.5, 0)$-sparse.
Matchings are a well-known example of something that is not a matroid.
Moreover, an odd cycle $C_{2k+1}$ is $(1,0)$-sparse, but for any partition $C_{2k+1} = M \cup G_2$ where $M$ is a matching, $G_2$ is not $(a_2, 0)$-sparse unless $a_2 \geq 2/3$.
Constructions of sharpness for the Nine Dragon Tree Conjecture \cite{MONTASSIER201238} give examples of graphs $G_{k,d}$ that are $(k + \frac{d}{k+d+1}, d+1- (k + \frac{d}{k+d+1}))$-sparse with $k,d \in \mathbb{Z}$, but any partition $G_{k,d} = T_k \cup F_d$ where $T_k$ is $(k,-k)$-sparse results in a graph $F_d$ whose maximum degree is at least $d+1$, and therefore not $(a_2,0)$-sparse unless $a_2 \geq \frac{d}{d+1}$.
Fan, Li, Song, and Yang \cite{FAN201572} modified this construction to produce $(k + \frac{d}{d+k+1},1)$-sparse graphs $H_{k,d}$ such that any partition $H_{k,d} = P_k \cup F_d$ where $P_k$ is $(k, 0)$-sparse results in a graph $F_d$ whose maximum degree is at least $d+1$.
Thus, by setting $k \gg d \gg 1$, we see that we can get as far away from the matroid case as possible: for arbitrary $\epsilon>0$ there exists an integral $a$ and a graph that is $(a+\epsilon,0)$-sparse and can not be partitioned into a $(a,0)$-sparse subgraph and a $(1-\epsilon,0)$-sparse subgraph.

Although it is messier, some things can be proven about the nonintegral case.
Inspired by the Nash-Williams theorem, the Strong Nine Dragon Tree Conjecture by Montassier, Ossona de Mendez, Raspaud, and Zhu \cite{MONTASSIER201238} asks if a $(k + \frac{d}{k+1+d}, -(k + \frac{d}{k+1+d}))$-sparse graph can be partitioned into a $(k,-k)$-sparse graph and a $(\frac{d}{d+1},0)$-sparse graph.
Montassier, Ossona de Mendez, Raspaud, and Zhu \cite{MONTASSIER201238} proved this case when $(d,k)=(1,1)$, Yang \cite{YANG201840} proved it for the remaining cases when $d=1$, the case $(k,d) = (1,2)$ was proven by Kim, Kostochka, West, Wu, and Zhu \cite{Kim2021}, and Mies and Moore \cite{mies2023strong} proved the $d \leq k+1$ case.
The Strong Nine Dragon Tree Conjecture remains open in general.
Chen, Kim, Kostochka, West and Zhu \cite{CHEN2017741} created the Overfull Nine Dragon Tree Conjecture, which contains a special case that asks if a graph that is both $(k + \frac{1}{k+2}, -k + \frac{2k}{k+2})$-sparse and $(k+1, -(k+1))$-sparse, then it can be partitioned into a $(k, -k)$-sparse and a $(1/2, 0)$-sparse graph.
The Overfull Nine Dragon Tree Conjecture was proved by Mies and Moore \cite{IGT_2024__1__21_0}.
Mies, Moore, and Smith-Roberge \cite{MIES2025104214} established a ``Pseudoforest Strong Nine Dragon Tree Theorem,'' in that a $(k + \frac{d}{d+k+1},0)$-sparse graph can be partitioned $P_k \cup F_d$ such that $P_k$ is $(k,0)$-sparse and $F_d$ is $(\frac{d}{d+1},0)$-sparse. 


This paper's primary focus is a conjecture by Kuperwasser, Samotij, and Wigderson \cite{KUPERWASSER_SAMOTIJ_WIGDERSON_2025}: that a $(m, 0)$-sparse graph $G$ can be partitioned $E(G) = E(F) \cup E(G')$ such that $F$ is $(1,-1)$-sparse and $G'$ is $(m, 1-2m)$-sparse.
A graph $G$ is strictly $(a,b)$-sparse if every subgraph $H$ satisfies $e(H) < a v(H) + b$.
As partial progress, Christoph, Martinsson, Steiner, and Wigderson \cite{Christoph} proved that we can have (A) $F$ is $(1,0)$-sparse and $G'$ is $(m, 1-2m)$-sparse, or (B) $F$ is $(1,-1)$-sparse and $G'$ is strictly $(m, -4m/3)$-sparse when $m > 1.5$.
Our main result is a full proof of the conjecture by Kuperwasser, Samotij, and Wigderson, which we present in Theorem \ref{main theorem}

\begin{theorem}\label{main theorem}
Let $m > 1$.
If $G$ is $(m, 0)$-sparse, then there exists a partition $G = F \cup G'$ such that $F$ is $(1,-1)$-sparse and $G'$ is $(m, 1-2m)$-sparse.
\end{theorem}

Let us give an intuitive outline for our approach; the following formulas are not correct and should be treated as back-of-the-envelope-approximations.
The key idea is that every simple graph of order at most $2n+1$ is $(n,0)$-sparse, which implies that for $\delta>0$ a $(a, b)$-sparse graph is also $(a + \delta, b - 2a\delta)$-sparse.
Throughout this paper let $k$ be integral and $0 \leq \epsilon < 1$ such that $k + \epsilon = m$ (e.g. $k = \lfloor m \rfloor$ and $\epsilon = m - k$).
The proof then follows three steps.
\begin{enumerate}
	\item Use the above idea to argue that a $(k + \epsilon, 0)$-sparse graph is also $(k+1, - 2m(1-\epsilon))$-sparse.
	\item Apply Theorem \ref{partitioning theorem} to step 1 to say that $G$ can be partitioned into a forest $F$ and a $(k, 1-2m(1-\epsilon))$-sparse graph $H$.
	\item Use the above idea a second time to say that $H$ is $(k+\epsilon, 1 - 2m(1-\epsilon) - 2m\epsilon)$-sparse.
\end{enumerate}
This approach will work without additional ideas when $k \geq 4$ and $1/2 \leq \epsilon \leq 3/4$.

The proof does require additional ideas to fix two problems, which we now quickly describe.
The first problem is that when $\epsilon$ is less than a half, step (2) may fail because the sparsity coefficients $(k, 1-2m(1-\epsilon))$ do not satisfy the conditions of Theorem \ref{partitioning theorem}.
In this situation, we observe that being $(k+\epsilon, 0)$-sparse is sufficiently stronger that we can find the desired partition $G \cup F$ directly, even when such partitions do not exist for $(k, 1-2m(1-\epsilon))$-sparse graphs. 
This is done in Theorem \ref{problematic partitioning}.
The second problem is that the most straight-forward approach to step (3) only proves that $H$ is $(m,4\epsilon-2-2m)$-sparse, which is insufficient when $\epsilon > 3/4$.
The solution here is that we can remove edges from the least-sparse regions of $H$ by swapping them with edges in $F$ in an operation similar to a matroid base exchange.
This idea is presented in Lemmas \ref{forest and sparse pseudoforest} and \ref{brooks-like decomposition}.

The outline of this paper is as follows.
The conjecture by Kuperwasser, Samotij, and Wigderson was motivated by an application in Ramsey theory for random graphs; we discuss that application in Section \ref{prob ramsey sec}.
We consider partitioning a graph by matroids in Section \ref{sec: partitioning}.
Theorem \ref{main theorem} for when $1 < m < 2$ is proven in Lemma \ref{small m theorem}; the arguments leading up to and including that result are contained to Section \ref{sec: small m}.
Theorem \ref{main theorem} for $m\geq 2$ is proven in Lemma \ref{main theorem for large m}; the arguments leading up to and including that result are contained in Section \ref{large m section}.

\subsection{Kohayakawa--Kreuter Conjecture} \label{prob ramsey sec}

A family of graphs $\mathcal{F}$ is a monotone property if $H \in \mathcal{F}$ and $H \leq G$ implies that $G \in \mathcal{F}$.
Let $G(n,p)$ denote the random graph on $n$ vertices with each edge occurring independently with probability $p$.
Let a.a.s. stand for ``asymptotically almost surely.''
Bollob{\'a}s and Thomason \cite{Bollobs1987ThresholdF} showed that every nontrivial monotone property $\mathcal{F}$ has a threshold function $t(n)$, where for any $q_0(n) \ll t(n) \ll q_1(n)$ we have that 
$G(n, q_0(n)) \notin \mathcal{F}$ and $G(n, q_1(n)) \in \mathcal{F}$ a.a.s.
The thresholds (conjectured or proven) we will discuss only require $T_{\mathcal{F}} q_0(n) < t(n) < q_1(n)/T_{\mathcal{F}}$ for a sufficiently large $T_{\mathcal{F}}$.

We say that a graph $G$ is Ramsey for a tuple of graphs $(H_1, H_2, \ldots, H_r)$ with $r \geq 2$ if for every partition $E(G) = E(G_1) \cup E(G_2) \cup \cdots \cup E(G_r)$ there exists an $i$ such that $H_i \leq G_i$.
Let $R_{H_1, \ldots, H_r}$ denote the set of graphs that are Ramsey for $(H_1, H_2, \ldots, H_r)$, which is a monotone property.
Let $m_2(G) = \max\{ \frac{|E(J)|-1}{|V(J)|-2} : J \leq G, |V(J)| \geq 3\}$.
Simple arithmetic gives that $m_2(G) = \min\{m : G \mbox{ is } (m, 1-2m)\mbox{-sparse}\}$.
R\"{o}dl and Ruci\'{n}ski \cite{rodl1995threshold} proved that for fixed $H$ with\footnote{If $m_2(H) \leq 1$, then $H$ is a forest, which is considered a pathological case.  Some results for forests are known, but some of them have threshold functions that differ from the function described here.} $m_2(H) > 1$ that for all $r \geq 2$ the function $n^{-1/m_2(H)}$ is a threshold for $R_{H, H, \ldots, H}$. 

Let $m(G) = \max\{ |E(J)|/|V(J)| : J \leq G\}$.
Simple arithmetic gives that $m(G) = \min\{m : G \mbox{ is } (m, 0)\mbox{-sparse}\}$.
Let $m_2(H_1, H_2) = \max\{ |E(J)|/(|V(J)| - 2 + 1/m_2(H_2)) : J \leq H_1\}$.
The Kohayakawa--Kreuter Conjecture \cite{Kohayakawa1997} asks if for $(H_1, \ldots, H_r)$ with $m_2(H_1) \geq m_2(H_2) \geq \cdots \geq m_2(H_r)$ and $m_2(H_2) > 1$, then is $n^{-1/m_2(H_1, H_2)}$ a threshold for $R_{H_1, H_2, \ldots, H_r}$.
Mousset, Nenadov, and Samotij \cite{Mousset_Nenadov_Samotij_2020} established $n^{-1/m_2(H_1, H_2)}$ is an upper bound on a threshold.
Later, Bowtell, Hancock, and Hyde\cite{bowtell2023proof} and independently Kuperwasser, Samotij, and Wigderson \cite{KUPERWASSER_SAMOTIJ_WIGDERSON_2025} gave a necessary and sufficient condition for the conjecture: $n^{-1/m_2(H_1, H_2)}$ is a lower bound for a threshold if and only if $G' \notin R_{H_1, H_2}$ whenever $m(G') \leq m_2(H_1, H_2)$.

As partial progress towards the full Kohayakawa--Kreuter Conjecture, Bowtell, Hancock, and Hyde\cite{bowtell2023proof} and Kuperwasser, Samotij, and Wigderson \cite{KUPERWASSER_SAMOTIJ_WIGDERSON_2025} describe ways to partition a $(m_2(H_1, H_2), 0)$-sparse graph $E(G') = E(G_1) \cup E(G_2)$ such that $H_1 \not\leq G_1$ and $H_2 \not\leq G_2$.
The different partitions follow the general idea of proving that $G_1$ and $G_2$ satisfy some condition that $H_1$ and $H_2$ do not.
For example, if $\chi(H_2) \geq 3$, then both groups independently found a partition where $G_2$ is bipartite and $G_1$ has degeneracy that is strictly less than the degeneracy of $H_1$.
A natural approach is to use sparsity as the condition: find some $a_1, a_2, b_1, b_2$ such that $G_i$ is $(a_i, b_i)$-sparse but $H_i$ is not.
This is precisely the motivation for Kuperwasser, Samotij, and Wigderson's conjecture that a $(m, 0)$-sparse graph can be partitioned into a $(1,-1)$-sparse graph and a $(m, 1-2m)$-sparse graph: by the assumption $m(H_2) > 1$ it follows that $H_2$ is not $(1,-1)$-sparse, and it is well-known (e.g., see \cite{KUPERWASSER_SAMOTIJ_WIGDERSON_2025}) that when $m_2(H_2) < m_2(H_1)$ it follows that $m_2(H_1, H_2) < m_2(H_1)$, and therefore $H_1$ is not  $(m_2(H_1, H_2), 1-2m_2(H_1, H_2))$-sparse.
(The case when $m_2(H_2) = m_2(H_1)$ had already been settled by Kuperwasser and Samotij \cite{Kuperwasser_Samotij_2024}.)

It is also the motivation for the progress made by Christoph, Martinsson, Steiner, and Wigderson, who proved their graph partition results are sufficient to establish the Kohayakawa--Kreuter Conjecture.
They do this by splitting into four cases.
If $m_2(H_1, H_2) \leq 1.5$, then they prove they can partition $G'$ into two forests, which suffices.
If $H_2$ is not a pseudoforest, then a partition of $G'$ into a $(1,0)$-sparse and a $(m, 1-2m)$-sparse subgraph suffices.
They also prove that if $H_2$ is a pseudoforest with girth at least $4$, then $H_2$ is not strictly $(m_2(H_1, H_2), -4m_2(H_1, H_2)/3)$-sparse, and thus a partition into a $(1,-1)$-sparse and a strictly $(m_2(H_1, H_2), -4m_2(H_1, H_2)/3)$-sparse subgraph suffices.
Finally, if $H_2$ is a pseudoforest that contains a triangle, then they appeal to the above partition into a bipartite subgraph and a degenerate subgraph.

Several extensions of the Kohayakawa--Kreuter Conjecture remain open, including the Kohayakawa--Kreuter Conjecture for Families.
We say that $G$ is Ramsey to a tuple of graph families $(\mathcal{H}_1, \ldots, \mathcal{H}_r)$ if for every partition $E(G) = E(G_1) \cup E(G_2) \cup \cdots \cup E(G_r)$ there exists an $i$ such that $H_i \leq G_i$ for some $H_i \in \mathcal{H}_i$.
We can define $R_{\mathcal{H}_1, \ldots, \mathcal{H}_r}$ in the analogous way, which is also a monotone property.
The Kohayakawa--Kreuter Conjecture for Families is that for $(\mathcal{H}_1, \ldots, \mathcal{H}_r)$ that avoid the pathological cases, a threshold for $R_{\mathcal{H}_1, \ldots, \mathcal{H}_r}$ is the minimum of thresholds for $R_{H_1, H_2, \ldots, H_r}$, where $H_i \in \mathcal{H}_i$.
The upper bound is easy to prove: the definitions imply $R_{H_1, \ldots, H_r} \subseteq R_{\mathcal{H}_1, \ldots, \mathcal{H}_r}$.

So it stands to reason that the challenge of the Kohayakawa--Kreuter Conjecture for Families will be in establishing the lower bound.
That Theorem \ref{main theorem} resolves the Kohayakawa--Kreuter Conjecture for Families follows from the fact that the two relevant results have been proven in the context of families.
That is, (1) Kuperwasser, Samotij, and Wigderson \cite{KUPERWASSER_SAMOTIJ_WIGDERSON_2025} showed that the lower bound is true if and only if $G' \notin R_{\mathcal{H}_1, \mathcal{H}_2}$ whenever $m(G) \leq \min_{H_1 \in \mathcal{H}_1, H_2 \in \mathcal{H}_2} m_2(H_1, H_2)$, and (2) Kuperwasser and Samotij \cite{Kuperwasser_Samotij_2024} proved the case for $\min_{H_1 \in \mathcal{H}_1} m_2(H_1) = \min_{H_2 \in \mathcal{H}_2}m_2(H_2)$.

However, the existing results do not resolve the conjecture for families.
While the Kohayakawa--Kreuter Conjecture asks for a structure in $G_i$ that is not in $H_i$, the Kohayakawa--Kreuter Conjecture for Families asks for a structure in $G_i$ that is not present in any member of $\mathcal{H}_i$.
For example, a bipartite graph may contain a copy of $C_6$, and careful examination of the definitions reveals that every $H_1$ is strictly $(m_2(H_1, C_3), -4m_2(H_1, C_3)/3)$-sparse.
So while the above arguments resolve each of the cases $\mathcal{H}_2 = \{C_6\}$ and $\mathcal{H}_2 = \{C_3\}$, the case $\mathcal{H}_2 = \{C_3, C_6\}$ remains open.

Kuperwasser, Samotij, and Wigderson propose sparsity as the most promising condition for being universally applicable to the graphs in some graph family $\mathcal{H}_i$; in the final paragraph of Section 1.4 they write that their conjecture is ``...the right way to resolve [the Kohayakawa--Kreuter Conjecture for Families] in its entirety...''
Therefore, with Theorem \ref{main theorem}, we have accomplished the last task in their approach to resolving the Kohayakawa--Kreuter Conjecture for Families.

\section{Partitioning Via Matroids}\label{sec: partitioning}

In the following we will give conditions for when a $(a, b)$-sparse graph can be partitioned into a $(a_1, b_1)$-sparse subgraph and a $(a_2, b_2)$-sparse subgraph using matroids.
We will only assume that $(E(G), \mathcal{I}_{a,b})$ is a matroid when $a,b$ are integral, and the Matroid Union Theorem.
The rank function $r$ of a matroid $(U, \mathcal{I})$ is defined as $r(A) = \max\{|I|: I \subset A, I \in \mathcal{I}\}$.
The Matroid Union Theorem (e.g. see \cite{West_2020}) states that for a finite set of matroids over a common set of elements $(U, \mathcal{I}_i)$ with rank functions $r_i$ there is a matroid $(U, \mathcal{I}_+)$ defined as $\mathcal{I}_+ = \{\cup_i I_i  : I_i \in \mathcal{I}_i\}$ whose rank function is $r_+(A) = \min_{B \subseteq A} \{ |A \setminus B| + \sum_i r_i(B)\}$.

Let $\mathcal{I}_i$ denote the graphs that are $(a_i, b_i)$-sparse, and let $r_i$ be the associated rank function.
Let $(E(G), \mathcal{I}_+)$ be the matroid formed from the union of $(E(G), \mathcal{I}_1)$ and $(E(G), \mathcal{I}_2)$.
So $G$ can be partitioned into a $(a_1, b_1)$-subgraph and a $(a_2, b_2)$-subgraph if and only if $r_+(E(G)) = |E(G)|$, which by the Matroid Union Theorem is equivalent to proving $r_1(E(B)) + r_2(E(B)) \geq e(B)$ for all subgraphs $B \leq G$.
This may seem like circular logic: to establish $r_1(E(B)) + r_2(E(B)) \geq e(B)$ all we have to work with is that $B$ is $(a_1 + a_2, b_1 + b_2)$-sparse.
The benefit is that for $E(G_1) \cup E(G_2)$ to be a partition of $E(G)$ we need that (1) $e(G_1) + e(G_2) \geq e(G)$ and (2) $E(G_1) \cap E(G_2) = \emptyset$; but the Matroid Union Theorem implies that it suffices to only satisfy (1) for some partition to exist.

Observe that while the above requires that $\mathcal{I}_1, \mathcal{I}_2, \mathcal{I}_+$ are matroids, it does not require that the family of $(a,b)$-sparse graphs form a matroid. 
This is good, because in Theorem \ref{problematic partitioning} we will want to show that $r_+(E(G)) = |E(G)|$ for a $G$ that satisfies a stronger sparsity condition with nonintegral coefficients (and therefore does not form a matroid).

We will begin by considering the easiest case, which is $0 \leq -b_i \leq a_i$.
The cases that we consider after will not contain full proofs, but instead simply remark on what parts of the easiest case need to be modified.
We also illustrate cases where a $(a_1 + a_2, b_1 + b_2)$-sparse graph can not be partitioned into a $(a_1, b_1)$-sparse subgraph and a $(a_2, b_2)$-sparse subgraph to highlight the sharpness of the results.

A graph is $(a, b)$-tight if\footnote{Some people also require that each proper subgraph is strictly $(a,b)$-sparse, but we do not.} it is both $(a,b)$-sparse and satisfies $e(G) = a v(G) + b$.  

\begin{counterex}\label{disconnected counterexample}
For integral $a_1, a_2$, there exists infinitely many graphs that are $(a_1 + a_2, 0)$-sparse that do not partition into a graph that is $(a_1, -1)$-sparse and a graph that is $(a_2, 1)$-sparse.
\end{counterex}
\begin{proof}
Let $H$ be a $2(a_1 + a_2)$-regular graph on $n$ vertices, which is $(a_1 + a_2, 0)$-tight.
Let $G$ be constructed from the union of $t$ disjoint copies of $H$, which is also $(a_1 + a_2, 0)$-sparse.
Let $G_1$ denote any $(a_1, -1)$-sparse subgraph of $G$, and let $E(G_2) = E(G) \setminus E(G_1)$.
As $G_1$ contains at most $a_1 n - 1$ edges in a copy of $H$, we have that $G_2$ contains at least $a_2 n + 1$ edges in each copy of $H$.
Therefore $G_2$ contains at least $a_2 |V(G)| + t$ edges overall, and thus is not $(a_2, 1)$-sparse when $t > 1$.
\end{proof}

\begin{counterex}\label{tree counterexample}
For integral $a,t$ with $1 \leq t < a$, there exists infinitely many graphs that are $(2a, -2a)$-sparse that do not partition into a graph that is $(a, -t)$-sparse and a graph that is $(a, t-2a)$-sparse.
\end{counterex}
\begin{proof}
Let $H$ be a graph that is a union of $2a$ disjoint spanning trees on $n$ vertices, which is $(2a, -2a)$-tight.
Let $H_1$ and $H_2$ be two copies of $H$, and let $G$ be formed by gluing one vertex of $H_1$ to one vertex of $H_2$.
As $G$ is also the union of $2a$ disjoint spanning trees, it is also $(2a, -2a)$-tight.
Let $G_2$ denote any $(a, t-2a)$-sparse graph of $G$, and let $E(G_1) = E(G) - E(G_2)$.
As $G_2$ contains at most $an + t - 2a$ edges in each of $H_i$ and $H_i$ has $2an-2a$ edges, we have that $G_1$ contains at least $an-t$ edges in each $H_i$. 
So 
$$ |E(G_1)| \geq 2(an-t) = a |V(G)| - t + (a-t), $$
and therefore $G_1$ is not $(a, -t)$-sparse.
\end{proof}

Intuitively speaking, a potential function provides a real valued score to each subgraph, where the score is a positive scalar times the number of vertices in the subgraph minus another positive scalar times the number of edges in the subgraph.
The score is called the ``potential'' of the subgraph.
Potential functions with appropriately chosen scalars are a strong tool for understanding a sparsity condition on a graph.
We will use potential functions throughout this paper with varying scalar coefficients to reflect the different sparsity conditions (e.g. $(k+1,b)$-sparse versus $(k,b-1)$-sparse).

One of the most important uses of potential functions is their submodularity.
That is for potential function $\rho$ and subgraphs $H_1, H_2$ of some common graph $G$ (so that operations like $H_1 \cap H_2$ and $H_1 \cup H_2$ are well-understood), then by counting how many times an edge or vertex shows up (and recalling that vertices contribute positively and edges contribute negatively to a linear function) we have that
$$ \rho(H_1 \cap H_1) + \rho(H_1 \cup H_2) \leq \rho(H_1) + \rho(H_2). $$
If $\rho(H) = a v(H) - e(H)$, then every nonempty subgraph $H$ of an $(a,b)$-sparse graph satisfies $\rho(H) \geq -b$, and $\rho(H) = -b$ if and only if $H$ is $(a,b)$-tight.
The term ``nonempty'' is subtle yet important here; the empty graph has potential zero and an isolated vertex has potential $a$, each of which can be less than $-b$.
In particular, in the following theorem the condition $a_i \geq -b_i \geq 0$ is required specifically so that all subgraphs with a vertex have potential at least $-b$.

\begin{theorem}\label{basic modularity theorem}[Theorem \ref{partitioning theorem}(A)]
For integral $a_i,b_i$ with $a_i \geq -b_i \geq 0$ for $i \in \{1,2\}$, if $G$ is $(a_1 + a_2, b_1 + b_2)$-sparse, then there exists a partition $E(G) = E(G_1) \cup E(G_2)$ such that $G_i$ is $(a_i, b_i)$-sparse for $i \in \{1, 2\}$.
\end{theorem}
\begin{proof}
By way of contradiction, let $B$ be the smallest set of edges in $G$ such that $r_1(B) + r_2(B) < |B|$.
The minimality of $B$ implies that for every edge $e \in B$ we have that $r_i(B \setminus \{e\}) = r_i(B)$ for $i \in \{1,2\}$.

Suppose that $B$ is disconnected; and let $B = B' \cup B''$ be a partition into disjoint subgraphs.
Because $b_i \leq 0$, the union of two disjoint $(a_i, b_i)$-sparse graphs is again $(a_i, b_i)$-sparse, and therefore $r_i(B') + r_i(B'') = r_i(B)$.
By minimality of $B$, we have that $r_1(B') + r_2(B') \geq |B'|$ and $r_1(B'') + r_2(B'') \geq |B''|$.
Therefore $r_1(B) + r_2(B) \geq |B|$, which contradicts the choice of $B$, and so we may assume $B$ is connected.

Let $B_i$ denote a maximal element of $\mathcal{I}_i$ that is contained in $B$, which by definition has size $r_i(B)$.
Let $W$ denote the set of vertices spanned by $B$, and let $H,H_i$ denote the subgraphs with vertex set $W$ and edge set $B, B_i$, respectively.
By construction, we have that $H_i$ is $(a_i, b_i)$-sparse and $H$ is $(a_1 + a_2, b_1 + b_2)$-sparse.

If $H_1$ and $H_2$ are both tight, then $|B_1| + |B_2| = (a_1 + a_2)|W| + (b_1 + b_2) \geq |B|$, and we are done.
So, by way of contradiction, suppose there exists some $i$ where $H_i$ is not tight.
Without loss of generality, let $i=1$.

By submodularity of the potential function $\rho(J) = a_1 v(J) - e(J)$, the union of two $(a_1, b_1)$-tight subgraphs whose intersection contains at least one vertex is again a $(a_1, b_1)$-tight subgraph.
This implies that maximal $(a_1, b_1)$-tight subgraphs are vertex disjoint.
So we can partition the vertex set $W = T_1 \cup T_2 \cup \cdots \cup T_r \cup L$ (possibly $r=0$ or $L = \emptyset$) where each $T_i$ is the vertex set of a maximal-sized $(a_1, b_1)$-tight subgraph of $H_1$, and every edge not contained in a $T_i$ is not part of any tight subgraph.
As $H_1$ is not tight, we have $r \geq 2$ or $L \neq \emptyset$.
As $B$ is connected, there exists some edge $e$ that is not contained in any single $T_j$.

Let $\overline{B_1} = B \setminus B_1$.
If $\overline{B_1} = \emptyset$, then $|B_1| = |B|$ and we are done, so suppose $f \in \overline{B_1}$.
By the maximality of $B_1$, we have that $B_1 + f$ is not $(a_1, b_1)$-sparse.
Because $a_1$ and $b_1$ are integral, this means that $f$ is contained in a $(a_1, b_1)$-tight subgraph $H_f$ of $B_1$.
This implies that each component of $\overline{B_1}$ is contained in some $T_i$, which implies that $e \in B_1$.

So now we argue that $r_1(B \setminus \{e\}) = r_1(B)-1$, which contradicts the choice of $B$.
Let $F_i$ denote the edges of $B$ that are contained by $V(T_i)$, and let $F_* = B - \cup_i F_i$.
Observe that $e \in F_*$, and more generally we argue that for any $F' \subseteq F_*$ we have $r_1(B \setminus F') = r_1(B)-|F'|$.
Let $B' \subseteq B \setminus F'$ be independent in $(E(G), \mathcal{I}_1)$ such that $|B'| = r_1(B \setminus F')$.
We bound $|B'|$ by bounding the size of its intersection with $F_1, \ldots, F_r, F_*$ and summing.
Specifically, we have $|B' \cap F_i| \leq a_1 |V(F_i)| + b_1 = |B_1 \cap F_i|$ and $|B' \cap F_*| \leq |F_*| - |F'| = |B_1 \cap F_*| - |F'|$, and so 
$r_1(B \setminus F') = |B'| \leq |B_1| - |F'| = r_1(B) - |F'|$.
\end{proof}

In the next theorem we prove that we can partition via matroids if both $b_1$ and $b_2$ are positive, and first we give an intuitive explanation.
In Counterexample \ref{disconnected counterexample} we see that if $b_1$ is positive, then the union of two disjoint $(a, b_1)$-tight subgraphs is not $(a, b_1)$-sparse.
But if $b_1$ and $b_2$ are both positive, then $b_1 + b_2$ is also positive, and therefore a $(a_1 + a_2, b_1 + b_2)$-sparse graph can not be formed from the union of several disjoint $(a_1 + a_2, b_1 + b_2)$-tight subgraphs.
So the fact that in our partition $E(G) = E(G_1) \cup E(G_2)$ the graph $G_1$ is forced to be strictly $(a_1, b_1)$-sparse in all but one component is canceled out by the fact that $G$ must also be strictly $(a_1 + a_2, b_1 + b_2)$-sparse in all but one component.

\begin{theorem}\label{negative modularity theorem}[Theorem \ref{partitioning theorem}(B)]
For integral $a_i,b_i$ with $b_i \geq 0$ for $i \in \{1,2\}$, if $G$ is $(a_1 + a_2, b_1 + b_2)$-sparse, then there exists a partition $E(G) = E(G_1) \cup E(G_2)$ such that $G_i$ is $(a_i, b_i)$-sparse.
\end{theorem}
\begin{proof}
We repeat almost all of the details from the proof of Theorem \ref{basic modularity theorem}.
The only place where $0 \geq b_i$ for $i \in \{1,2\}$ is used in that proof is to show that $B$ is connected.
So this theorem will follow if we can simply argue that $0 \leq b_i$ for $i \in \{1,2\}$ also implies that $B$ is connected.

In order for this proof to work, we modify the proof to use induction on the size of $G$.
Suppose that $B$ is disconnected; and let $B = B' \cup B''$ be a partition into disjoint subgraphs.
Let $b'$ be the minimum integer such that $B'$ is $(a_1 + a_2, b')$-sparse, and define $b''$ similarly.
In particular, some subgraph of $B'$ is $(a_1 + a_2, b')$-tight and some subgraph of $B''$ is $(a_1 + a_2, b'')$-tight.
If either $b'$ or $b''$ is negative, then replace it with zero.

We argue that $b' + b'' \leq b_1 + b_2$. 
As $G$ is $(a_1 + a_2, b_1 + b_2)$-sparse, we have $b', b'' \leq b_1 + b_2$.
So we are done if $b' = 0$ or $b'' = 0$; this implies that there exists a subgraph $\widehat{B'}$ that is $(a_1+a_2,b')$-tight and a subgraph $\widehat{B''}$ that is $(a_1+a_2,b'')$-tight.
Let $\widehat{B} = \widehat{B'} \cup \widehat{B''}$, and observe that $\widehat{B}$ is $(a_1 + a_2, b' + b'')$-tight.
As $\widehat{B}$ is part of a $(a_1 + a_2, b_1 + b_2)$-sparse graph, it follows that $b' + b'' \leq b_1 + b_2$, as desired.

So we can find nonnegative integers $b_1', b_2', b_1'', b_2''$ such that 
\begin{itemize}
	\item $b_i' + b_i'' \leq b_i$ for $i \in \{1,2\}$,
	\item $b_1' + b_2' = b'$, and 
	\item $b_1'' + b_2'' = b''$.
\end{itemize}
By induction on the size of $G$, we have partitions $E(B') = E(H_1') \cup E(H_2')$ and $E(B'') = E(H_1'') \cup E(H_2'')$, where $H_i'$ and $H_i''$ are $(a_i, b_i')$-sparse and $(a_i, b_i'')$-sparse, respectively.
Observe that $H_i' \cup H_i''$ is $(a_i, b_i' + b_i'')$-sparse, and therefore $r(B) \geq |B|$, which contradicts the choice of $B$.
\end{proof}

In this paper we take special consideration for partitioning into two graphs when one of them is a forest.

\begin{counterex}\label{ring example}
For integral $a$, there exists infinitely many graphs that are $(a+1,-a-2)$-sparse that do not partition into subgraph that are $(1,-1)$-sparse and $(a,-a-1)$-sparse.
\end{counterex}
\begin{proof}
Let $H$ be a $(a+1, -a-2)$-tight graph with distinct vertices $x,y$.
An example of such a graph is a complete graph on $2a+2$ vertices with one edge removed.
Let $H_1, \ldots , H_t$ be copies of $H$ with $t \geq a+2$, where $H_i$ has vertices $x_i, y_i$.
Construct $G$ from the union of the $H_i$ and gluing $y_i$ to $x_{i+1}$ for each $i < t$ and gluing $y_t$ to $x_1$.

We claim that $G$ is $(a+1, -a-2)$-sparse.
Let $J$ be a subgraph of $G$, and define potential function $\rho(J) = (a+1)v(J) - e(J)$; our goal is to show that $\rho(J) \geq a+2$.
By submodularity we can assume that $J$ is connected.
Let $J_i$ be the subgraph of $J$ that is contained in $H_i$.
By the sparsity of $H$, we have that $\rho(J_i) \geq a+2$ for each $i$.
Because $J$ is connected, we can assume without loss of generality that $J_i$ is nonempty for $1 \leq i \leq t'$ for some $t'$.
If $t' < t$, then applying submodularity gives that 
\begin{eqnarray*}
	\rho(J)	& = & \rho(J_1) + \sum_{i=2}^{t'} (\rho( \cup_{1 \leq j \leq i} J_j ) - \rho( \cup_{1 \leq j < i} J_j )) \\
		& \geq & (a+2) + \sum_{i=2}^{t'} (\rho(J_i) - \rho(\{x_i\}) ) \\
		& \geq & (a+2) + \sum_{i=2}^{t'} (a+2 - (a+1)) \geq a+2.
\end{eqnarray*}
Now we consider the case $t' = t$.
We will apply a similar argument, but we must deal with the fact that $J_t \cap (\cup_{j < t} J_i)$ could contain two vertices ($x_t$ and $x_1$) instead of just one.
Let $J'' = J_t \cap (\cup_{j < t} J_i)$, and observe that $\rho(J'') \leq 2(a+1)$.
So, using $t \geq a+2$, we have 
\begin{eqnarray*}
	\rho(J)	& = & \rho(J_1) + (\rho(J) - \rho(\cup_{j < t}J_j)) + \sum_{i=2}^{t-1} (\rho( \cup_{1 \leq j \leq i} J_j ) - \rho( \cup_{1 \leq j < i} J_j )) \\
		& \geq & (a+2) + (\rho(J_t) - \rho(J'')) + \sum_{i=2}^{t-1} (\rho(J_i) - \rho(\{x_i\}) ) \\
		& \geq & (a+2) + (a+2 - 2(a+1)) + (t - 2)(a+2 - (a+1)) \geq a+2.
\end{eqnarray*}
This proves the claim.

Next we show that $G$ can not be partitioned into $F$ and $G_2$, where $F$ is a forest and $G_2$ is $(a,-a-1)$-sparse.
Suppose that $H$ has $n$ vertices.
Because $H$ is $(a+1, -a-2)$-tight, this means that it has exactly $(a+1)(n-1) - 1$ edges.
So $G$ has $t(n-1)$ vertices and $t((a+1)(n-1)-1)$ edges.
As $G_2$ is $(a,-a-1)$-sparse, it contains at most $a(n-1)-1$ edges in each $J_i$, and therefore it has $t(a(n-1)-1)$ edges overall.
As $F$ is a forest, it as at most $t(n-1)-1$ edges.
As $t(n-1)-1 + t(a(n-1)-1) < t((a+1)(n-1)-1)$, the proof is complete.
\end{proof}

Our fantasy was to prove that a $(k+\epsilon,0)$-sparse graph is $(k+1, b)$-sparse, and therefore it can be partitioned into a forest and a graph that is $(k,b+1)$-sparse.
This is Theorem \ref{basic modularity theorem} when $-b \leq k+1$, but by Counterexample \ref{ring example} (and generalizations of it) we know this is impossible when $-b > k+1$.
So instead we now must give less elegant statements, where we show that while $(k+1, b)$-sparse graphs can not be partitioned, the strictly smaller family of $(k+\epsilon,0)$-sparse graphs can be.
In the following we use $f(k, \epsilon)$ to be the formal definition for our intuitive value $-b$.
Our definition of $f$ caps the value of $f(k, \epsilon)$ at $2k$, which matches the threshold for a pathological matroid at $(k, 1-2k)$-sparsity.

Let 
\begin{equation}
f(k, \epsilon) = 
\begin{cases}
 2k 							& \text{if } \epsilon(2k+2) < 2 \\
 \lceil (2k+2)(1-\epsilon) \rceil 			& \text{if } \epsilon \in [2/(2k+2), 1/2) \\
 k+1							& \text{if } \epsilon \in [1/2, (k+2)/(2k+3) )\\
 \lceil (2k+3)(1-\epsilon) \rceil			& \text{if } \epsilon \geq (k+2)/(2k+3)
\end{cases}
\end{equation}

\begin{lemma} \label{lemma defining f}
If $G$ is $(k+\epsilon, 0)$-sparse, then $G$ is $(k+1, -f(k, \epsilon))$-sparse.
\end{lemma}
\begin{proof}
We will work with a potential function $\rho$ over subgraphs $H \subseteq G$ with $v = |V(H)|$ and $e = |E(H)|$ defined as $\rho(H) = (k+1)v - e$.
Our goal is to prove $\rho(H) \geq f(k, \epsilon)$ for all $H$ with $v \geq 2$.
By construction, $f(k, \epsilon) \leq 2k$.  

As $G$ is simple, we have that $e \leq \min\{ \lfloor m v \rfloor, {v \choose 2}\}$.
Observe that as a function over $t$, the polynomial $\rho(K_t) = (k+1)t - {t \choose 2}$ is quadratic and concave.
Moreover, $\rho(K_2) = \rho(K_{2k+1}) = 2k+1$, so we may assume $v \geq 2k+2$.

\textbf{Case 1:} $\epsilon < 1/2$.
We have that $e \leq \lfloor (k+\epsilon)v \rfloor$, so
\begin{eqnarray*}
 \rho(H) 	& \geq 	& (k+1) v - \lfloor (k+\epsilon)v \rfloor \\
 		& = 	& \lceil v (1-\epsilon) \rceil \geq \lceil (2k+2) (1-\epsilon) \rceil.
\end{eqnarray*}
Observe that $\epsilon (2k+2) < 2$ if and only if $ \lceil (2k+2) (1-\epsilon) \rceil > 2k$.

\textbf{Case 2:} $\epsilon \geq 1/2$.
If $\epsilon \geq 1/2$, then $K_{2k+2}$ is $(k+\epsilon,0)$-sparse, where $\rho(K_{2k+2}) = k+1$.
Observe that $\epsilon \geq \frac{k+2}{2k+3}$ if and only if $k+1 \geq  \lceil (2k+3)(1-\epsilon) \rceil$.
So now we may assume that $v \geq 2k+3$.
We have that $e \leq \lfloor (k+\epsilon)v \rfloor$, so
$$ \rho(H) \geq \lceil v (1-\epsilon) \rceil \geq \lceil (2k+3) (1-\epsilon) \rceil. $$
\end{proof}

\begin{prop}\label{hypergraph size claim}
Let $\mathcal{H}$ be a hypergraph with hyperedges $F_1, F_2, \ldots, F_r$ that span $n$ vertices.
If there is an $s$ such that for each $i$ we have $| F_i \cap \cup_{j \neq i} F_j| \geq s$, then $\sum_i |F_i| \geq n + r s / 2$.
\end{prop}
\begin{proof}
We prove this by discharging.
The initial charge $c_1: V(\mathcal{H}) \cup E(\mathcal{H}) \rightarrow \mathbb{R}$ for each vertex is equal to its degree and for each hyperedge is zero. 
During the discharging phase each vertex keeps one charge for itself, and gives any remaining charge uniformly among its incident edges.
Each vertex with degree at least two gives at least $0.5$ charge to each incident edge.
So the final charge $c_2: V(\mathcal{H}) \cup E(\mathcal{H}) \rightarrow \mathbb{R}$ for each vertex is one and for each hyperedge is at least $s/2$.
So we conclude that
$$ \sum_i |F_i| = \sum_v d(v) = \sum_{x \in V \cup E} c_1(x) = \sum_{x \in V \cup E} c_2(x) \geq n + r s /2.$$
\end{proof}

\begin{theorem}\label{problematic partitioning}
If $k$ is positive intergral and $\epsilon \in [0, 1)$, then a $(k + \epsilon, 0)$-sparse graph can be partitioned into a forest and a $(k, 1-f(k,\epsilon))$-sparse graph.
\end{theorem}
\begin{proof}
Suppose $G$ is $(k+\epsilon,0)$-sparse.
We apply Lemma \ref{lemma defining f} to say that $G$ is also $(k+1,-f(k,\epsilon))$-sparse.
We are done by Theorem \ref{basic modularity theorem} if $f(k,\epsilon) \leq k+1$, so assume $f(k, \epsilon) > k+1$.
In particular we have that $\epsilon < 1/2$ and so $f(k, \epsilon) \leq (2k+2)(1-\epsilon)$.
Also, because $f(k,\epsilon) \leq 2k$ by definition, we have that $k \geq 2$. 

We repeat the opening arguments in the proof to Theorem \ref{basic modularity theorem}.
Recall that $B$ is a minimal set of edges in $G$ such that $|B| > r_1(B) + r_2(B)$, $H$ is the subgraph with $E(H) = B$ and $V(H) = W$, $B_1$ is a maximal forest contained in $B$, and $B_2$ is a maximum sized subgraph contained in $B$ that is $(k,1-f(k,\epsilon))$-sparse.
The only place where $-b_i \leq a_i$ is used is to prove that the union of two $(a_i, b_i)$-tight subgraphs with nontrivial intersection is again a tight subgraph, and so everything before that still holds.
In particular, we have that $H$ is connected, so $B_1$ is a spanning tree, which implies $r_1(B) = |W| - 1$.
The goal for the rest of this proof is to contradict the choice of $B$ by showing that $|B_2| \geq |B| - |W| + 1$.

We claim that $H$ is $2$-connected.
By way of contradiction, suppose $W = W_1 \cup W_2$ with $W_1 \cap W_2 = \{x\}$ and $B = E(H[W_1]) \cup E(H[W_1])$ with each $H[W_i]$ being nonempty.
The minimality of $B$ implies that $H[W_i]$ can be partitioned into $F_i \cup G^{(i)}$, where $F_i$ is a forest and $G^{(i)}$ is $(k, 1-f(k,\epsilon))$-sparse.
By construction $F = F_1 \cup F_2$ is also a forest.
Let $\rho_2(J) = k v(J) - e(J)$ be the potential function such that $(k, 1-f(k,\epsilon))$-sparsty is equivalent to $\rho_2(J) \geq f(k, \epsilon) -1$ for all nonempty subgraphs $J$.
By construction of $W_1, W_2$, it follows that $G^{(1)} \cap G^{(2)}$ is empty or the isolated vertex $x$, and therefore it has potential $0$ or $k$.
Let $J_i$ be an arbitrary nonempty subgraph of $G^{(i)}$; submodularty and $f(k, \epsilon) > k+1$ imply that 
$$ \rho_2( J_1 \cup J_2) \geq \rho_2(J_1) + \rho_2(J_2) - \rho_2(J_1 \cap J_2) \geq 2(f(k,\epsilon)-1) - k \geq f(k, \epsilon)-1. $$
Therefore $G^{(1)} \cup G^{(2)}$ is $(k, 1-f(k,\epsilon))$-sparse, which contradicts that $B$ does not have a partition.
This proves the claim.

Let $T_1, \ldots, T_r$ be the set of maximal subgraphs of $B_2$ that are $(k, 1-f(k, \epsilon))$-tight.
As in the proof to Theorem \ref{basic modularity theorem}, by the maximality of $B_2$, every edge in $B \setminus B_2$ is contained in a $T_i$ for some $i$.
So every edge not in some $T_i$ is in $B_2$.
Using submodularity, if $T_i \cap T_j$ contains at least two vertices, then $T_i \cup T_j$ is also $(k, 1-f(k, \epsilon))$-tight, which contradicts their maximality.
So every edge of $B$ is contained in one or zero $T_i$.

For any $J \subseteq B_2 \subset B$, we have that $|B_2| \geq |B| - |W| + 1$ if and only if $|B_2 \setminus J| \geq |B \setminus J| - |W| + 1$.
We will choose $J$ to be the edges that are not in a $T_i$ or in a $T_i$ that is isomorphic to $K_2$.
Note that if $T_i$ has order at least three, then $T_i$ is still tight after the removal of $J$, as the tight subgraphs are edge-disjoint.
If $T_i$ is isomorphic to $K_2$, then because $G$ is simple, it contains no edge of $B \setminus B_2$, so every edge of $B \setminus B_2$ is in a tight subgraph of $B_2 \setminus J$.
Thus, without loss of generality, we can assume that every edge of $B$ is contained by some unique $T_i$ and that $T_i$ is a $(k, 1-f(k, \epsilon))$-tight subgraph of $B_1$ with order at least three.

Let $n_i = v(T_i)$; by definition $e(T_i) = k n_i + 1 - f(k, \epsilon)$.
Let $n = |W|$ and $n' = \sum_{i} n_i - n$, so $|B_2| = k (n + n') + r (1 - f(k, \epsilon))$.
By Claim \ref{hypergraph size claim} and the $2$-connectivity of $H$, we have $n' \geq r$.
As $H$ is $(k+\epsilon, 0)$-sparse, we have that $|B| \leq (k + \epsilon) n$.

We claim that $n_i \geq 2k-1$ for each $i$, which implies that $(2 k - 1) r \leq n + n'$.
As we already established that $n_i \geq 3$ for each $i$, we may assume $k \geq 3$ for proving the claim.
By definition, we have $\rho_2(T_i) = f(k,\epsilon)-1 \leq 2k-1$.
Observe that $\rho_2(K_2) = 2k-1$ and $\rho_2(K_{2k-2}) = 3k-3 > 2k-1$.
As $\rho_2(K_n)$ is quadratic and concave function of $n$ and $n_i \geq 3$, this implies that $T_i$ is not a subgraph of $K_n$ for $n \leq 2k-2$.
By integrality, this proves the claim.

Using the all of the above inequalities, we have 
\begin{eqnarray*}
|B_2| - |B| + |W| - 1	& =	& k (n + n') + r (1 - f(k, \epsilon)) - | B| + n - 1 \\
			& \geq 	& k(n + n') + r(1 - (2k+2)(1-\epsilon)) - (k + \epsilon)n + n - 1\\
			& = 	& (1 - \epsilon)(n - (2 k - 1) r) + kn' + r(1 - 3 (1-\epsilon)) - 1 \\
			& \geq	& (1-\epsilon)(-n') + kn' + r(1 - 3(1-\epsilon)) -1 \\
			& =	& (k-1)n' - 2r - 1 + \epsilon(n' + 3r) \\
			& \geq 	& (k-1)n' - 2r - 1 \\
			& \geq 	& (k-3)r - 1.\\
\end{eqnarray*}			
If $k > 3$, then we have finished the proof to the theorem.
So all that is left is to repeat the above calculation after using $k \in\{2,3\}$ to establish a tighter bound on $f(k, \epsilon)$ than $(2k+2)(1-\epsilon)$.

\textbf{Case $k=3$:}
As $f(k,\epsilon) > k+1 = 4$ and $f(k, \epsilon) \leq 2k = 6$, we have that $f(k, \epsilon) \in \{5,6\}$.
By definition, $f(3, \epsilon) = 5$ when $1/4 < \epsilon < 3/8$ and $f(3, \epsilon) = 6$ when $\epsilon \leq 1/4$.

First let us consider $1/4 < \epsilon < 3/8$ so 
\begin{eqnarray*}
|B_2| - |B| + |W| - 1	& \geq	& 3 (n + n') + r (1 - 5) - (3 + \epsilon)n + n - 1 \\
			& = 	& (1-\epsilon)(n+n') + (2+\epsilon)n'  -4 r - 1\\
			& \geq  & (1-\epsilon)(5r) + (2+\epsilon)n' - 4r - 1 \\
			& \geq	& (3 - 4\epsilon) r - 1 \\
			& \geq 	& (3/2)(3) - 1 > 0.
\end{eqnarray*}	
So now consider $\epsilon \leq 1/4$ so 
\begin{eqnarray*}
|B_2| - |B| + |W| - 1	& \geq	& 3 (n + n') + r (1 - 6) - (3 + \epsilon)n + n - 1 \\
			& = 	& (1-\epsilon)(n+n') + (2+\epsilon)n'  - 5 r - 1\\
			& \geq  & (1-\epsilon)(5r) + (2+\epsilon)n' - 5 r - 1 \\
			& \geq	& (2 - 4\epsilon) r - 1 \\
			& \geq 	& (1)(3) - 1 > 0.
\end{eqnarray*}	

\textbf{Case $k=2$:}
As $f(k,\epsilon) > k+1 = 3$ and $f(k, \epsilon) \leq 2k = 4$, we have that $f(k, \epsilon) = 4$.
By definition, $f(2, \epsilon) = 4$ when $\epsilon \leq 1/3$.
So we have
\begin{eqnarray*}
|B_2| - |B| + |W| - 1	& \geq	& 2 (n + n') + r (1 - 4) - (2 + \epsilon)n + n - 1 \\
			& = 	& (1-\epsilon)(n+n') + (1+\epsilon)n'  -3 r - 1\\
			& \geq  & (1-\epsilon)(3r) + (1+\epsilon)n' - 3r - 1 \\
			& \geq	& (1 - 2\epsilon) r - 1 \\
			& \geq 	& (1/3)(3) - 1 \geq 0.
\end{eqnarray*}	
\end{proof}

\section{Small $m$}\label{sec: small m}

\begin{remark}\label{from k to m}
If $G$ is $(m', 1-2m')$-sparse for some $m' \leq m$, then $G$ is $(m, 1-2m)$-sparse.
\end{remark}
\begin{proof}
When $v \geq 2$, $m' v + 1-2m' \leq m v + 1 - 2m$.
\end{proof}

\begin{lemma}\label{forest and sparse pseudoforest}
If $G$ is $(2,-1)$-sparse, then we can partition $G = F \cup P$, where $F$ is a forest and $P$ is a triangle-free pseudoforest.
\end{lemma}
\begin{proof}
We will call a partition $G = F \cup P$ ``valid'' if $F$ is $(1,-1)$-sparse and $P$ is $(1, 0)$-sparse.
By Theorem \ref{partitioning theorem} a valid partition exists, and we expect many to exist.
Our aim is to prove there exists a valid partition such that $P$ contains no triangles.
We will do this by replacing $F, P$ with valid partition $F', P'$ that contains one fewer triangle; the lemma follows from repeating this process until no triangles remain.

Suppose $P[W] \cong K_3$ for $W = \{x, y, z\}$.
For each $v \in W$, let $e_v$ denote the edge of $P$ whose endpoints are $W \setminus \{v\}$.
If the endpoints of any edge $e_v$ are in different components of $F$, then $F' = F + e$ and $P' = P - e$ suffices.  
So there exists a path $Q_v \subseteq F$ whose endpoints are $e_v$.

It is a property of forests that for any edge $e \in Q_v$, we have that $F - e + e_v$ is a forest.
If $P - e_v + e$ is a pseudoforest with fewer triangles than $P$, then we are done by setting $F' = F - e + e_v$ and $P' = P - e_v + e$.
So for every $e \in Q_v$ we have that adding $e$ to $P-e_v$ creates (1) a component with at least two cycles or (2) a new triangle.

Let $C$ denote the component of $P$ that contains $W$.
Observe that $P[C]-e_v$ is a tree.
So if $e$ has exactly one endpoint in $C$, then adding $e$ to $P-e_v$ does not do (1) or (2) above.
Therefore $e$ has both or neither endpoints in $C$.
As both endpoints of $Q_v$ are in $C$, it follows that $P_v \subseteq C$.
As $P[C]-e_v$ is a tree, there are no edges $e \in Q_v$ that do (1).
So every edge $e \in Q_v$ does (2), which is to say that there exists a vertex $w_e \in C$ such that $e \subseteq N_P(w_e)$.

Let $Q_v = z_1 z_2 \ldots z_t$ where $\{z_1, z_t\} \subset W$, and let $x_i = w_{z_i z_{i+1}}$.
Consider the closed walk $v z_1 x_1 z_2 x_2 \ldots x_{t-1} z_t v$, which does not contain the edge $e_v$ but is contained in $P$.
As each component of $P$ contains at most one cycle, and that cycle in $C$ is $xyz$, the above walk does not contain a cycle.
As the $z_i$ do not repeat because they form the path $Q_v$, it follows that $x_1 = x_2 = \cdots = x_{t-1} = v$.
Therefore $Q_v \subseteq N_P(v)$.

As $F$ is a forest, it follows that $Q_x \cup Q_y \cup Q_z$ forms a subdivided $K_{1,3}$.
As each $v \in W$ is adjacent to everything in two of the three subdivided edges of the subdivided $K_{1,3}$, the elements of $W$ have at least one common neighbor outside of $W$.
Let $u \notin W$ such that $|N_P(u) \cap W| \geq 2$; the subgraph $P[\{x,y,z,u\}]$ has at least $5$ edges and $4$ vertices and therefore is not $(1,0)$-sparse, which contradicts our choice of $P$.
\end{proof}

\begin{lemma}\label{small m theorem}
Theorem \ref{main theorem} is true when $1 < m < 2$.
\end{lemma}
\begin{proof}
Suppose $G$ is $(m,0)$-sparse, where $1< m< 2$.
By Lemma \ref{lemma defining f}, $G$ is $(2,-1)$-sparse, and $(2,-2)$-sparse if $m<1.8$.
So if $m < 1.8$, then by Theorem \ref{partitioning theorem} there exists a partition $G = F \cup F'$, where each of $F$ and $F'$ is $(1,-1)$-sparse.
By Remark \ref{from k to m}, we are done by setting $G' = F'$.

So now assume $1.8 \leq m < 2$ and $G$ is $(2,-1)$-sparse.
Let $G = F \cup P$ be the partition given by Lemma \ref{forest and sparse pseudoforest}.
We claim that $P$ is $(m, 1-2m)$-sparse, which will finish the proof.

By the lemma we have that $P$ is $(1,0)$-sparse and triangle-free.
Let $H$ be a subgraph of $P$, and let $v = |V(H)|$ and $e = |E(H)|$.
Being triangle-free means that if $v=3$, then $e \leq 2$.
Thus, it suffices to prove that $mv - e - 2m + 1\geq 0$ when $v \geq 4$.
As $P$ is $(1,0)$-sparse we have that $e \leq v$, so 
$$mv - e - 2m + 1\geq (m-1)v - 2m + 1\geq 2m - 3 \geq 3.6 -3 > 0.$$
\end{proof}


\section{Large $m$}\label{large m section}

Recall $m = k + \epsilon$, where $k$ is integral and $0 \leq \epsilon < 1$.

\begin{lemma}\label{brooks-like decomposition}
If $G$ is $(k+1,-s)$-sparse with $1 \leq s \leq k-1$, then we can partition $G = F \cup G'$, where $F$ is a forest, $G'$ is $(k,1-s)$-sparse, and every subgraph $H'$ of $G'$ spanning exactly $2k+1$ vertices contains at most $(2k+1)k + s$ edges.  
\end{lemma}
\begin{proof}
Let $G, k, b$ be as in the statement of the lemma.
We will call a partition $G = F \cup \widehat{G}$ ``valid'' if $F$ is $(1,-1)$-sparse (and therefore a forest) and $\widehat{G}$ is $(k, 1-s)$-sparse.
By Theorem \ref{partitioning theorem} a valid partition exists, and we expect many to exist.
Our aim is to prove there exists a valid partition that also satisfies that every subgraph $\widehat{H}$ of $\widehat{G}$ spanning exactly $2k+1$ vertices contains at most $(2k+1)k - s$ edges.
By being $(k, 1-s)$-sparse, $\widehat{H}$ contains at most $(2k+1)k + 1-s$ edges, and so we need to improve this bound by one.

We will work with a potential function $\widehat{\rho}$ over vertex sets $U$ and subgraphs $\widehat{H} \subseteq \widehat{G}$ as $\widehat{\rho}(U, \widehat{H}) = k |U| - |E(\widehat{H}[U])|$.
As $\widehat{G}$ is $(k,1-s)$-sparse we have that $\widehat{\rho}(U, \widehat{G}) \geq s-1$ for all $U$.
If every $U$ with $|U| = 2k+1$ satisfies $\widehat{\rho}(U, \widehat{G}) \geq s$, then we are done.
Fix a valid partition $F \cup \widehat{G}$ and some $U$ with $|U| = 2k+1$ and $\widehat{\rho}(U, \widehat{G}) = s-1$, and let $\widehat{H} = \widehat{G}[U]$.  

For a valid partition $F' \cup G'$, define $S_{F', G'}$ to be the family of vertex sets $W$ such that $|W| = 2k+1$ and $\widehat{\rho}(G'[W]) = s-1$.
We have that $U \in S_{F, \widehat{G}}$.
The lemma is equivalent to the existence of a valid partition such that $S_{F', G'} = \emptyset$.
We claim that there exists a valid partition $F^* \cup G^*$ such that $S_{F^*, G^*} \subseteq S_{F, \widehat{G}} \setminus \{U\}$.
The lemma would then follow from iteratively applying the claim, and so the rest of this proof is dedicated to proving the claim.

Observe that $\widehat{\rho}(K_1, K_1) = \widehat{\rho}(K_{2k}, K_{2k}) = k$.
The function $kt - {t \choose 2}$ is quadratic and concave in $t$, so every nonempty subgraph with at most $2k$ vertices has potential at least $k$.
As $|U| = 2k+1$, this means every nonempty subgraph of $\widehat{H}$ that is missing at least one vertex has potential at least $k$.

Pick an edge $e = xy \in E(\widehat{G})$.
If $x$ and $y$ are in different components of $F$, then $F^* = F + e$ and $G^* = \widehat{G} - e$ suffices.
So we may assume there exists a path $P \subseteq F$ with endpoints $x$ and $y$.
Let $f$ be the edge in $P$ incident with $y$.
We wish to specifically pick a $y$ such that $f \not\subseteq U$.
As $G$ is simple, the number of edges of $F$ contained in $U$ is at most 
$${|U| \choose 2} - |E(\widehat{H})| = {2k + 1 \choose 2} - (k(2k+1) + 1 -s) = s-1 \leq k-2 . $$
Therefore at most $2k-4$ vertices are incident with an edge in $F[U]$.
As $|U| = 2k+1$, we can use the pigeon hole principle to find an edge $xy$ such that $y$ is not incident with an edge in $F[U]$, and therefore the other endpoint of $f$ is outside $U$.
We will show that the claim follows from setting $F^* = F + e - f$ and $G^* = \widehat{G} + f - e$ for this choice of $e,f$.

We need to show that $F^* \cup G^*$ is a valid partition and that $S_{F^*, G^*} \subseteq S_{F, \widehat{G}} \setminus \{U\}$.
By construction, $F^*$ is a forest.
As $f \not\subseteq U$ and $e \subseteq U$, we have that $\widehat{\rho}(G^*, U) = \widehat{\rho}(\widehat{G}, U) + 1$, and so $U \notin S_{F^*, G^*}$.
So both of the things we wish to show follow from the following statement:
\begin{center}
for any $W$, if $\widehat{\rho}(G^*,W) < \widehat{\rho}(\widehat{G},W)$, then $\widehat{\rho}(G^*,W) \geq s$.
\end{center}
To see why the statement is true, suppose that $W$ is a vertex set with $\widehat{\rho}(G^*,W) < \widehat{\rho}(\widehat{G},W)$.
This implies that $\widehat{\rho}(G^*,W) = \widehat{\rho}(\widehat{G},W)-1$ with $f \subseteq W$ and $e \not\subseteq W$.
In particular, $W \cap \{y,x\} = \{y\}$.
Therefore $\emptyset \neq W \cap U \subseteq U-x$.
By the above, this implies $\widehat{\rho}(\widehat{G},W \cap U) \geq k$.

Applying submodularity we have 
\begin{eqnarray*}
s-1						& \leq		& \widehat{\rho}(\widehat{G}, W \cup U) \\
 				 		& \leq		& \widehat{\rho}(\widehat{G}, W) + \widehat{\rho}(\widehat{G}, U) - \widehat{\rho}(\widehat{G}, W \cap U) \\
						& \leq		& \widehat{\rho}(\widehat{G}, W) + (s-1) - k \\
						& \leq		& \widehat{\rho}(\widehat{G}, W) - 2.
\end{eqnarray*}						
So $\widehat{\rho}(\widehat{G}, W) \geq s+ 1$, and so $\widehat{\rho}(G^*,W) \geq (s+1) - 1 = s$.
This concludes the proof to the lemma.
\end{proof}

\begin{lemma}\label{main theorem for large m}
Theorem \ref{main theorem} is true when $m > 2$.
\end{lemma}
\begin{proof}
Suppose $G$ is $(m,0)$-sparse, where $m = k + \epsilon$, $k$ is integral, $k \geq 2$, and $0 \leq \epsilon < 1$.
By Theorem \ref{problematic partitioning} there exists a partition $G = F \cup G'$ such that $F$ is $(1,-1)$-sparse and $G'$ is $(k, 1-f(k, \epsilon))$-sparse.

We will work with a potential function $\rho'$ over subgraphs $H' \subseteq G'$ with $v = |V(H')|$ and $e = |E(H')|$ defined as $\rho'(H') = (k+\epsilon)v - e$.
Our goal is to show that $\rho'(H') \geq 2m-1$ for all $H'$ with $v \geq 2$.
Let $g(x) = mx - x(x-1)/2$ for $x \in \mathbb{R}$, so $g(t) = \rho'(K_t)$ for integral $t$.
Because $g(2) = g(2m-1) = 2m-1$, we may assume that $v > 2m-1$.

We now break into cases by the value of $\epsilon$.
These cases are not split along the inequalities in the definition of $f$, but instead upon the level sets of $f$.

\textbf{Case A:} $\epsilon < \frac{3}{2k+2}$.
Observe that in this case we have that $f(k, \epsilon) \geq 2k$, so $G'$ is $(k, 2k-1)$-sparse.
We are then done by Remark \ref{from k to m}.

\textbf{Case B:} $\frac{3}{2k+2} < \epsilon < 1/2$.
Recall that $v \geq \lceil 2m-1 \rceil = 2k$.
By Lemma \ref{lemma defining f} we have $f(k, \epsilon) = \lceil (2k+2)(1-\epsilon) \rceil$, so $e \leq k v + 1 - \lceil (2k+2)(1-\epsilon) \rceil$.
Therefore 
\begin{eqnarray*}
  \rho'(H')	& \geq & (k+\epsilon)v - \left( k v +1 - \lceil (2k+2)(1-\epsilon) \rceil  \right) \\
  		& = & 2k+1 + \epsilon v - \lfloor (2k+2)\epsilon \rfloor \\
  		& \geq & 2k + 1 + (2k)\epsilon -  (2k+2)\epsilon  \\
  		& = & 2k + 1 - 2 \epsilon > 2k + 2\epsilon -1 = 2m-1.
\end{eqnarray*} 

This concludes Case B.
In the remaining cases we have $\epsilon \geq 1/2$, and so we may assume $v > 2m-1 \geq 2k$, so $v \geq 2k+1$.

\textbf{Case C:} $1/2 \leq \epsilon < (k+3)/(2k+3)$.
Observe that in this case we have that $f(k, \epsilon) = k+1$, so $e \leq k v - k$.
Therefore 
$$ \rho'(H') \geq \epsilon v + k \geq \epsilon(2k+1) + k \geq 2k + 1/2.$$ 
Because $k \geq 2$, we have that $\epsilon < (k+3)/(2k+3) \leq 5/7$, so $2k + 1/2 \geq 2k + 2\epsilon - 1 = 2m-1$.

\textbf{Case D:} $(k+3)/(2k+3) \leq \epsilon$.
By Lemma \ref{lemma defining f} we have $f(k, \epsilon) = \lceil (2k+3)(1-\epsilon) \rceil$, so $e \leq k v + 1 - \lceil (2k+3)(1-\epsilon) \rceil$.
If $v \geq 2k+2$, then 
\begin{eqnarray*}
  \rho'(H')	& \geq & (k+\epsilon)v - \left( k v + 1 - \lceil (2k+3)(1-\epsilon) \rceil  \right) \\
  		& = & 2k+2 + \epsilon v - \lfloor (2k+3)\epsilon \rfloor \\
  		& \geq & 2k + 2 + (2k+2)\epsilon -  (2k+3)\epsilon .  \\
  		& = & 2k + 2 - \epsilon > 2k + 2\epsilon - 1 = 2m-1.
\end{eqnarray*} 

So we may assume that $v = 2k+1$.
By repeating the above calculation (while reducing our bound on $v$ by $1$), we get that $\rho'(H') \geq 2k + 2 - 2\epsilon$.
We now proceed by considering several cases.

\textbf{Subcase D.1:} $\epsilon \leq 3/4$.
In this subcase we have $\rho'(H') \geq 2k + 2\epsilon -1 = 2m-1$, and we are done.

\textbf{Subcase D.2:} $\epsilon \geq \frac{k+4}{2k+3}$.
Observe that in this case we have that $f(k, \epsilon) \leq k-1$.
We now choose our partition $G = F \cup G'$ to be the partition provided by Lemma \ref{brooks-like decomposition}.
Specifically, this means that subgraphs on exactly $2k+1$ vertices have one fewer edges than provided by sparsity.
So 
$$ \rho'(H') \geq (2k + 2 - 2\epsilon) + 1 > 2k + 2\epsilon - 1 = 2m-1,$$
and we are done.

\textbf{Subcase D.3:} $3/4 < \epsilon <  \frac{k+4}{2k+3}$.
Observe that $\frac{k+4}{2k+3} < 3/4$ when $k \geq 4$, so we may assume $k \in \{2,3\}$.
Specifically, we have that 
\begin{itemize}
	\item $k=2$, $\epsilon \in (3/4, 6/7)$, which implies $f(k, \epsilon) = 2$, or
	\item $k=3$, $\epsilon \in (3/4, 7/9)$, which implies $f(k, \epsilon) = 3$.  
\end{itemize}

So if $k=2$, then $v = 2k+1 = 5$, $e = kv - f(k,\epsilon)+1 = 9$, and $2m-1 = 3 + 2\epsilon$.
So we have that (recall $\epsilon > 3/4$)
$$ \rho'(H') = (2+\epsilon)5 - 9 = 1 + 5 \epsilon > 3 + 2\epsilon = 2m-1.$$

And if $k=3$, then $v = 7$, $e = 19$, and $2m-1 = 5 + 2\epsilon$.
So we have that 
$$ \rho'(H') = (3 + \epsilon)7 - 19 = 2 + 7 \epsilon > 5 + 2 \epsilon = 2m-1.$$
\end{proof}

\bibliographystyle{alpha}
\bibliography{sparseForests}

\end{document}